\documentclass[12pt]{article}
\usepackage{amssymb}
\usepackage{amsthm}
\usepackage{amsmath}
\usepackage{graphicx}
\usepackage{enumerate}
\DeclareMathOperator{\Const}{Const}

\DeclareMathOperator{\Int}{Int}

\DeclareMathOperator{\Der}{Der}

\newtheorem{theorem}{Theorem}%[section]
\newtheorem{definition}[theorem]{Definition}
\newtheorem{lemma}[theorem]{Lemma}
\newtheorem{proposition}[theorem]{Proposition}
\newtheorem{remark}[theorem]{Remark}
\newtheorem{example}[theorem]{Example}
\newtheorem{corollary}[theorem]{Corollary}
\title{Integration in semirings}
\author{Ivan~Chajda and Helmut~L\"anger}
\date{}
\begin{document}

\footnotetext{Support of the research by the Austrian Science Fund (FWF), project I~4579-N, and the Czech Science Foundation (GA\v CR), project 20-09869L, entitled ``The many facets of orthomodularity'', as well as by \"OAD, project CZ~02/2019, entitled ``Function algebras and ordered structures related to logic and data fusion'', and, concerning the first author, by IGA, project P\v rF~2021~030, is gratefully acknowledged.}

\maketitle

\begin{abstract}
The concept of integral as an inverse to that of derivation was already introduced for rings and recently also for lattices. Since semirings generalize both rings and bounded distributive lattices, it is natural to investigate integration in semirings. This is our aim in the present paper. We show properties of such integrals from the point of view of semiring operations. Examples of semirings with derivation where integrals are introduced are presented in the paper. These illuminate rather specific properties of such integrals. We show when the set of all integrals on a given semiring forms a semiring again.
\end{abstract}



The concept of derivation in rings is known for many decades, see e.g.\ \cite{Ba} and references therein. This concept was generalized for lattices by G.~Sz\'asz (\cite S), see also e.g.\ \cite F. For semilattices it was generalized by J.~S.~Golan and H.~Simmons (\cite{GS}), see also the monograph \cite G and the papers \cite{CT}, \cite D and \cite{SG}.

The concept of integration as an inverse to that of derivation was introduced for rings by I.~Bani\v c (\cite{Ba}). Recently it was extended also for lattices in \cite{YAZ}. Because both rings and bounded distributive lattices are special cases of semirings as defined in the monograph \cite G, it is natural to extend this concept also to semirings what we will do here.

Although derivations on semirings were already introduced and treated in \cite{CT} and \cite D and applied for regular events in \cite{Br}, for the reader's convenience we repeat all necessary definitions.

Recall that a {\em {\rm(}unitary{\rm)} semiring} is an algebra $\mathbf S=(S,+,\cdot,0,1)$ of type $(2,2,0,0)$ such that
\begin{enumerate}[(i)]
\item $(S,+)$ is a commutative semigroup,
\item $(S,\cdot)$ is a semigroup,
\item $x\cdot(y+z)\approx x\cdot y+x\cdot z$ and $(y+z)\cdot x\approx y\cdot x+z\cdot x$,
\item $x\cdot0\approx0\cdot x\approx0$ and $x\cdot1\approx1\cdot x\approx x$.
\end{enumerate}
If $\mathbf S$ satisfies the identity $x\cdot y\approx y\cdot x$ then it is called {\em commutative}. We call an element $a$ of $S$ {\em additively invertible} or {\em invertible} if it possesses an {\em additive} or {\em multiplicative inverse}, respectively. Because of the associativity of $+$ and $\cdot$, such an inverse is unique. We will denote it by $-a$ or $a^{-1}$, respectively. Sometimes we consider elements that have only a one-sided inverse. A mapping $d:S\rightarrow S$ is called a {\em derivation} on $\mathbf S$ if
\begin{enumerate}[(i)]
\item $d(x+y)\approx d(x)+d(y)$,
\item $d(x\cdot y)\approx d(x)\cdot y+x\cdot d(y)$.
\end{enumerate}

The formula describing the derivation of a product of two elements can be easily extended to a formula describing the derivation of a finite product of elements.

\begin{lemma}\label{lem5}
Let $\mathbf S=(S,+,\cdot,0,1)$ be a semiring, $d\in\Der\mathbf S$, $n$ a positive integer and $a_1,\ldots,a_n,a\in S$. Then
\[
d(a_1\cdot\cdots\cdot a_n)=d(a_1)\cdot a_2\cdot\cdots\cdot a_n+a_1\cdot d(a_2)\cdot a_3\cdot\cdots\cdot a_n+\cdots+a_1\cdot\cdots a_{n-1}\cdot d(a_n)
\]
and hence $d(a^n)=n\cdot a^{n-1}\cdot d(a)$ if $\mathbf S$ is commutative.
\end{lemma}

\begin{proof}
We use induction on $n$. For $n=1$ the assertion is trivial. Now assume $n$ to be a positive integer and the formula to hold for $n$ factors. If $a_1,\ldots,a_{n+1}\in S$ then
\begin{align*}
d(a_1\cdot\cdots\cdot a_{n+1}) & =d\big((a_1\cdot\cdots\cdot a_n)\cdot a_{n+1}\big)=d(a_1\cdot\cdots\cdot a_n)\cdot a_{n+1}+a_1\cdot\cdots\cdot a_n\cdot d(a_{n+1})= \\
& =\big(d(a_1)\cdot a_2\cdot\cdots\cdot a_n+\cdots+a_1\cdot\cdots\cdot a_{n-1}\cdot d(a_n)\big)\cdot a_{n+1}+ \\
& \hspace*{5mm}+a_1\cdot\cdots\cdot a_n\cdot d(a_{n+1})= \\
& =d(a_1)\cdot a_2\cdot\cdots\cdot a_{n+1}+\cdots+a_1\cdot\cdots\cdot a_n\cdot d(a_{n+1}),
\end{align*}
i.e.\ the formula holds for $n+1$ factors. The last formula can be considered as the special case $a_1=\cdots=a_n:=a$.
\end{proof}

Let $\Der\mathbf S$ denote the set of all derivations on $\mathbf S$.

\begin{definition}
Let $\mathbf S=(S,+,\cdot,0,1)$ be a semiring, $d\in\Der\mathbf S$ and $a,b\in S$. Then $a$ is called a {\em $d$-integral} of $b$ if $d(a)=b$. Let $I_d(b)$ denote the set of all $d$-integrals of $b$. The element $b$ is called {\em $d$-integrable} if $I_d(b)\neq\emptyset$. The $d$-integrals of $0$ we call $d$-constants. Let {\em $\Const_d(S)$} denote the set of all $d$-constants of $S$. The element $b$ of $S$ is called {\em $d$-integrable} if $I_d(b)\neq\emptyset$. Let $\Int_d(S)$ denote the set of all $d$-integrable elements of $S$.
\end{definition}

Hence $a\in S$ is a $d$-constant if and only if $d(a)=0$. Moreover, it is easy to see that if $a$ is additively invertible then so is $d(a)$ and $d(-a)=-d(a)$.

A nice non-trivial example of a semiring with derivation $d$ where the $d$-integral has interesting properties is in the following one.

\begin{example}\label{ex1}
Let $\mathbf R=(R,+,\cdot,0,1)$ be a semiring and
\[
\mathbf S:=(S,+,\cdot,\left(
\begin{array}{cc}
0 & 0 \\
0 & 0
\end{array}
\right),\left(
\begin{array}{cc}
1 & 0 \\
0 & 1
\end{array}
\right))
\]
with
\[
S:=\{\left(
\begin{array}{cc}
x & y \\
0 & x
\end{array}
\right)\mid x,y\in R\}
\]
the semiring of $2\times2$-matrices of this form. It is easy to check that $S$ is closed with respect to addition and multiplication. Define
\[
d(\left(
\begin{array}{cc}
x & y \\
0 & x
\end{array}
\right)):=\left(
\begin{array}{cc}
0 & y \\
0 & 0
\end{array}
\right)\text{ for all }x,y\in R.
\]
Then {\rm(}see e.g.\ {\rm\cite G)} $d\in\Der\mathbf S$. We have
\begin{align*}
\Int_d(S) & =\{\left(
\begin{array}{cc}
0 & y \\
0 & 0
\end{array}
\right)\mid y\in R\}, \\
I_d(\left(
\begin{array}{cc}
0 & y \\
0 & 0
\end{array}
\right)) & =\{\left(
\begin{array}{cc}
x & y \\
0 & x
\end{array}
\right)\mid x\in R\}\text{ for all }y\in R, \\
\Const_d(S) & =\{\left(
\begin{array}{cc}
x & 0 \\
0 & x
\end{array}
\right)\mid x\in R\}, \\
\Int_d(S)\cap\Const_d(S) & =\{\left(
\begin{array}{cc}
0 & 0 \\
0 & 0
\end{array}
\right)\}.
\end{align*}
\end{example}

Another example which is close to the classical definition of derivation and integral in calculus, but has slightly different properties is the following one.

\begin{example}\label{ex2}
Let $\mathbf F[x]=(F[x],+,\cdot,0,1)$ be the semiring of polynomials in one variable over a field $\mathbf F=(F,+,\cdot,0,1)$. Define
\[
d(\sum_{i=0}^na_ix^i):=\sum_{i=1}^nia_ix^{i-1}
\]
for all $\sum\limits_{i=0}^na_ix^i\in F[x]$. Then $d\in\Der\mathbf F[x]$ and we have
\begin{align*}
                  \Int_d(F[x]) & =\left\{
\begin{array}{ll}
F[x]                                                               & \text{if }{\rm char}\mathbf F=0, \\
\{\sum\limits_{i=0}^na_ix^i\mid a_i=0\text{ if }i\equiv-1\bmod p\} & \text{if }{\rm char}\mathbf F=p,
\end{array}
\right. \\
I_d(\sum\limits_{i=0}^na_ix^i) & =\{\sum\limits_{i=0}^n\frac{a_i}{i+1}x^{i+1}+b\mid b\in F\}\text{ for all }\sum\limits_{i=0}^na_ix^i\in\Int_d(F[x]), \\
                \Const_d(F[x]) & =\left\{
\begin{array}{ll}
F                                                                     & \text{if }{\rm char}\mathbf F=0, \\
\{\sum\limits_{i=0}^na_ix^i\mid a_i=0\text{ if }i\not\equiv0\bmod p\} & \text{if }{\rm char}\mathbf F=p,
\end{array}
\right. \\
\Const_d(F[x]) & \subseteq\Int_d(F[x]).
\end{align*}
\end{example}

Let us recall some elementary, but useful properties of derivations on semirings.

Let $\mathbf S=(S,+,\cdot,0,1)$ be a semiring and $d_1,d_2\in\Der\mathbf S$. Then
\begin{align*}
     (d_1+d_2)(x+y) & =d_1(x+y)+d_2(x+y)=\big(d_1(x)+d_1(y)\big)+\big(d_2(x)+d_2(y)\big)= \\
                    & =\big(d_1(x)+d_2(x)\big)+\big(d_1(y)+d_2(y)\big)=(d_1+d_2)(x)+(d_1+d_2)(y), \\
(d_1+d_2)(x\cdot y) & =d_1(x\cdot y)+d_2(x\cdot y)= \\
                    & =\big(d_1(x)\cdot y+x\cdot d_1(y)\big)+\big(d_2(x)\cdot y+x\cdot d_2(y)\big)= \\
                    & =\big(d_1(x)\cdot y+d_2(x)\cdot y\big)+\big(x\cdot d_1(y)+x\cdot d_2(y)\big)= \\
							      & =\big(d_1(x)+d_2(x)\big)\cdot y+x\cdot\big(d_1(y)+d_2(y)\big)= \\
										& =(d_1+d_2)(x)\cdot y+x\cdot(d_1+d_2)(y).
\end{align*}
Hence $d_1+d_2\in\Der\mathbf S$. Trivially the mapping $d_0:S\rightarrow S$ defined by $d_0(x):=0$ for all $x\in S$ is a derivation on $\mathbf S$. Thus $(\Der\mathbf S,+,d_0)$ forms a monoid.

We start with some elementary but basic properties of $d$-integrals and sets of $d$-integrals on a given semiring.

\begin{proposition}\label{prop1}
Let $\mathbf S=(S,+,\cdot,0,1)$ be a semiring, $d\in\Der\mathbf S$ and $a,b\in S$. Then
\begin{enumerate}[{\rm(i)}]
\item $I_d(a)+I_d(b)\subseteq I_d(a+b)$,
\item if $e\in I_d(a)$ and $f\in I_d(b)$ then $e\cdot f\in I_d(a\cdot f+e\cdot b)$,
\item $0\in\Const_d(S)$,
\item if $1\in I_d(a)$ then $a+a=a$,
\item $d(a)\in\Int_d(S)$, $a\in I_d\big(d(a)\big)$, and if $a\in\Int_d(S)$ then $d\big(I_d(a)\big)=\{a\}$,
\item if $c,1\in\Const_d(S)$ and $c$ in invertible then $c^{-1}\in\Const_d(S)$.
\end{enumerate}
\end{proposition}

\begin{proof}
If $e\in I_d(a)$ and $f\in I_d(b)$ then $d(e)=a$ and $d(f)=b$ and hence
\begin{enumerate}[(i)]
\item $d(e+f)=d(e)+d(f)=a+b$, i.e.\ $e+f\in I_d(a+b)$,
\item $d(e\cdot f)=d(e)\cdot f+e\cdot d(f)=a\cdot f+e\cdot b$, i.e.\ $e\cdot f\in I_d(a\cdot f+e\cdot b)$.
\end{enumerate}
Moreover, we have
\begin{enumerate}
\item[(iii)] $d(0)=d(0\cdot0)=d(0)\cdot0+0\cdot d(0)=0$.
\item[(iv)] $a+a=d(1)\cdot1+1\cdot d(1)=d(1\cdot1)=d(1)=a$.
\item[(v)] This follows from the definition of $I_d\big(d(a)\big)$ and from the $d$-integrability of $a$.
\item[(vi)] If $c,1\in\Const_d(S)$ and $c$ in invertible then 
\begin{align*}
d(c^{-1}) & =1\cdot d(c^{-1})=(c^{-1}\cdot c)\cdot d(c^{-1})=c^{-1}\cdot\big(c\cdot d(c^{-1})\big)=c^{-1}\cdot\big(0+c\cdot d(c^{-1})\big)= \\
          & =c^{-1}\cdot\big(0\cdot c^{-1}+c\cdot d(c^{-1})\big)=c^{-1}\cdot\big(d(c)\cdot c^{-1}+c\cdot d(c^{-1})\big)= \\
          & =c^{-1}\cdot d(c\cdot c^{-1})=c^{-1}\cdot d(1)=c^{-1}\cdot0=0,
\end{align*}
i.e.\ $c^{-1}\in\Const_d(S)$.
\end{enumerate}
\end{proof}

It is easy to check that $I_{d_0}(0)=S$ and $I_{d_0}(x)=\emptyset$ for all $x\in S\setminus\{0\}$. Hence, $\Const_{d_0}(S)=S$.

\begin{remark}
Let $\mathbf S=(S,+,\cdot,0,1)$ be a semiring and $d\in\Der\mathbf S$. It is evident that $d$ is surjective if and only if $I_d(x)\neq\emptyset$ for all $x\in S$, i.e.\ if and only if every element of $S$ is $d$-integrable. On the other hand, $d$ is injective if and only if $|I_d(x)|\leq1$ for all $x\in S$ which means that every $d$-integrable element of $S$ has exactly one $d$-integral.
\end{remark}

In the sequel we state several natural results concerning the behavior of constants in $d$-integrals. The following lemma says that the product of a $d$-integrable element of $S$ and a $d$-constant of $S$ is again $d$-integrable.

\begin{lemma}\label{lem2}
Let $\mathbf S=(S,+,\cdot,0,1)$ be a semiring, $d\in\Der\mathbf S$, $a\in S$ and $c\in\Const_d(S)$. Then
\begin{enumerate}[{\rm(i)}]
\item $I_d(a)+c\subseteq I_d(a)$,
\item $c\cdot I_d(a)\subseteq I_d(c\cdot a)$ and $I_d(a)\cdot c\subseteq I_d(a\cdot c)$.
\end{enumerate}
\end{lemma}

\begin{proof}
If $b\in I_d(a)$ then
\begin{enumerate}[(i)]
\item $d(b+c)=d(b)+d(c)=a+0=a$,
\item $d(c\cdot b)=d(c)\cdot b+c\cdot d(b)=0\cdot b+c\cdot a=0+c\cdot a=c\cdot a$, \\
$d(b\cdot c)=d(b)\cdot c+b\cdot d(c)=a\cdot c+b\cdot0=a\cdot c+0=a\cdot c$.
\end{enumerate}
\end{proof}

Item (ii) of Lemma~\ref{lem2} implies that $c\cdot\Int_d(S)\subseteq\Int_d(S)$. A similar result does not hold for addition instead of multiplication. If we consider the semiring $(S,+,\cdot,0,1)$ with derivation $d$ of Example~\ref{ex1} with non-trivial $R$ then we have $c+a\notin\Int_d(S)$ for all $c\in\Const_d(S)\setminus\{0\}$ and $a\in\Int_d(S)$.

More interesting results can be obtained if we consider invertible elements of $S$.

\begin{lemma}\label{lem3}
Let $\mathbf S=(S,+,\cdot,0,1)$ be a semiring, $d\in\Der\mathbf S$ and $a,b,e\in S$ and assume $e$ to be invertible. Then
\begin{align*}
a\in I_d\big(e\cdot d(b)\big) & \Leftrightarrow b\in I_d\big(e^{-1}\cdot d(a)\big), \\
a\in I_d\big(d(b)\cdot e\big) & \Leftrightarrow b\in I_d\big(d(a)\cdot e^{-1}\big),
\end{align*}
\end{lemma}

\begin{proof}
$a\in I_d\big(e\cdot d(b)\big)$ is equivalent to $d(a)=e\cdot d(b)$, and $b\in I_d\big(e^{-1}\cdot d(a)\big)$ is equivalent to $d(b)=e^{-1}\cdot d(a)$. The second equivalence follows analogously.
\end{proof}

\begin{corollary}
Let $\mathbf S=(S,+,\cdot,0,1)$ be a semiring, $d\in\Der\mathbf S$ and $c\in\Const_d(S)$ and assume $c$ to be inverible. Then $c^{-1}\in I_d\big(c^{-1}\cdot d(1)\big)\cap I_d\big(d(1)\cdot c^{-1}\big)$.
\end{corollary}

\begin{proof}
Since $c^{-1}\in I_d\big(d(c^{-1})\big)$ we have $1=c\cdot c^{-1}\in I_d\big(c\cdot d(c^{-1})\big)$ according to (ii) of Lemma~\ref{lem2} whence $c^{-1}\in I_d\big(c^{-1}\cdot d(1)\big)$ because of Lemma~\ref{lem3}. The assertion $c^{-1}\in I_d\big(d(1)\cdot c^{-1}\big)$ follows analogously.
\end{proof}

Converse inclusions to those of Lemma~\ref{lem2} can be proved if we assume that the constant $c$ in question is additively invertible or invertible, respectively.

\begin{proposition}\label{prop2}
Let $\mathbf S=(S,+,\cdot,0,1)$ be a semiring, $d\in\Der\mathbf S$ and $a,b\in S$. Then {\rm(i)} -- {\rm(iii)} hold:
\begin{enumerate}[{\rm(i)}]
\item If $b$ is additively invertible and $-b\in\Const_d(S)$ then $I_d(a)\subseteq I_d(a)+b$,
\item if $b$ is invertible and $b^{-1}\in\Const_d(S)$ then
\begin{align*}
I_d(b\cdot a) & \subseteq b\cdot I_d(a), \\
I_d(a\cdot b) & \subseteq I_d(a)\cdot b.
\end{align*}
\end{enumerate}
\end{proposition}

\begin{proof}
\
\begin{enumerate}[(i)]
\item If $b$ is additively invertible and $-b\in\Const_d(S)$ then according to (i) of Lemma~\ref{lem2} we have
\[
I_d(a)=I_d(a)+0=I_d(a)+\big((-b)+b\big)=\big(I_d(a)+(-b)\big)+b\subseteq I_d(a)+b.
\]
\item If $b$ is invertible and $b^{-1}\in\Const_d(S)$ then according to (ii) of Lemma~\ref{lem2} we have
\begin{align*}
I_d(b\cdot a) & =1\cdot I_d(b\cdot a)=(b\cdot b^{-1})\cdot I_d(b\cdot a)=b\cdot\big(b^{-1}\cdot I_d(b\cdot a)\big)\subseteq \\
              & \subseteq b\cdot I_d\big(b^{-1}\cdot(b\cdot a)\big)=b\cdot I_d\big((b^{-1}\cdot b)\cdot a\big)=b\cdot I_d(1\cdot a)=b\cdot I_d(a), \\
I_d(a\cdot b) & =I_d(a\cdot b)\cdot1=I_d(a\cdot b)\cdot(b^{-1}\cdot b)=\big(I_d(a\cdot b)\cdot b^{-1}\big)\cdot b\subseteq \\
              & \subseteq I_d\big((a\cdot b)\cdot b^{-1}\big)\cdot b=I_d\big(a\cdot(b\cdot b^{-1})\big)\cdot b=I_d(a\cdot1)\cdot b=I_d(a)\cdot b.							
\end{align*}
\end{enumerate}
\end{proof}

If we assume that $c$ to be (additively) invertible then we obtain the following result:

\begin{corollary}
Let $\mathbf S=(S,+,\cdot,0,1)$ be a semiring, $d\in\Der\mathbf S$ and $a\in S$ and $c\in\Const_d(S)$. Then
\begin{enumerate}[{\rm(i)}]
\item If $c$ is additively invertible and $-c\in\Const_d(S)$ then $I_d(a)+c=I_d(a)$,
\item if $c$ is invertible and $c^{-1}\in\Const_d(S)$ then
\begin{align*}
I_d(c\cdot a) & =c\cdot I_d(a), \\
I_d(a\cdot c) & =I_d(a)\cdot c.
\end{align*}
\end{enumerate}
\end{corollary}

\begin{proof}
This follows from Lemma~\ref{lem2} and Proposition~\ref{prop2}.
\end{proof}

If a constant $c$ and has a one-side inverse then we can derive the following $d$-integration rules.

\begin{proposition}
Let $\mathbf S=(S,+,\cdot,0,1)$ be a semiring, $d\in\Der\mathbf S$, $a,b,e\in S$ and $c\in\Const_d(S)$. Then
\begin{enumerate}[{\rm(i)}]
\item If $c\cdot a=1$ then $c\cdot I_d(a\cdot b)\subseteq I_d(b)$,
\item if $c\cdot a=1$ and $a\cdot e\in I_d(b)$ then $e\in I_d(c\cdot b)$,
\item if $a\cdot c=1$ then $I_d(b\cdot a)\cdot c\subseteq I_d(b)$,
\item if $a\cdot c=1$ and $e\cdot a\in I_d(b)$ then $e\in I_d(b\cdot c)$.
\end{enumerate}
\end{proposition}

\begin{proof}
\
\begin{enumerate}[(i)]
\item If $c\cdot a=1$ then
\[
c\cdot I_d(a\cdot b)\subseteq I_d\big(c\cdot(a\cdot b)\big)=I_d\big((c\cdot a)\cdot b\big)=I_d(1\cdot b)=I_d(b)
\]
according to (i) of Lemma~\ref{lem2}.
\item If $c\cdot a=1$ and $a\cdot e\in I_d(b)$ then $e=1\cdot e=(c\cdot a)\cdot e=c\cdot(a\cdot e)\in I_d(c\cdot b)$ according to (i) of Lemma~\ref{lem2}.
\item follows analogously to (i).
\item follows analogously to (ii).
\end{enumerate}
\end{proof}

\begin{lemma}
Let $\mathbf S=(S,+,\cdot,0,1)$ be a semiring, $d\in\Der\mathbf S$ and $a\in S$ and assume $d\big(d(a)\big)=d(a)$ and $d(a)$ to be additively invertible. Then $a\in\Const_d(S)+\Int_d(S)$.
\end{lemma}

\begin{proof}
If $e:=a+\big(-d(a)\big)$ and $b:=d(a)$ then $b$ is $d$-integrable,
\begin{align*}
d(e) & =d\Big(a+\big(-d(a)\big)\Big)=d(a)+d\big(-d(a)\big)=d(a)+\Big(-d\big(d(a)\big)\Big)= \\
     & =d(a)+\big(-d(a)\big)=0, \\
 e+b & =\Big(a+\big(-d(a)\big)\Big)+d(a)=a+\Big(\big(-d(a)\big)+d(a)\Big)=a+0=a.
\end{align*}
\end{proof}

Observe that $d\big(d(c)\big)=d(0)=0=d(c)$ for all $c\in\Const_d(S)$.

In the next lemma we prove a result similar to ``integration per parts'' known in calculus.

\begin{lemma}
Let $\mathbf S=(S,+,\cdot,0,1)$ be a semiring, $d\in\Der\mathbf S$ and $a,b\in S$ and assume $e\in I_d(a)$, $f\in I_d\big(e\cdot d(b)\big)$ and $f$ to be additively invertible. Then $e\cdot b+(-f)\in I_d(a\cdot b)$.
\end{lemma}

\begin{proof}
We have
\begin{align*}
d\big(e\cdot b+(-f)\big) & =d(e\cdot b)+d(-f)=\big(d(e)\cdot b+e\cdot d(b)\big)+\big(-d(f)\big)= \\
                         & =a\cdot b+\Big(e\cdot d(b)+\big(-e\cdot d(b)\big)\Big)=a\cdot b+0=a\cdot b.
\end{align*}
\end{proof}

Finally, we can show when the addition of two sets of $d$-integrals is again a set of $d$-integrals.

\begin{theorem}
Let $\mathbf S=(S,+,\cdot,0,1)$ be a semiring, $d\in\Der\mathbf S$ and $a,b\in S$ and assume there exists some $e\in I_d(a)$ which is additively invertible.Then $I_d(a+b)=I_d(a)+I_d(b)$.
\end{theorem}

\begin{proof}
According to (i) of Proposition~\ref{prop1}, $I_d(a)+I_d(b)\subseteq I_d(a+b)$. Now we prove the converse inclusion. If $f\in I_d(a+b)$ and $g:=(-e)+f$ then
\begin{align*}
d(g) & =d\big((-e)+f\big)=d(-e)+d(f)=-d(e)+d(f)=(-a)+(a+b)= \\
     & =\big((-a)+a\big)+b=0+b=b
\end{align*}
and hence $g\in I_d(b)$ and
\[
e+g=e+\big((-e)+f\big)=\big(e+(-e)\big)+f=0+f=f.
\]
\end{proof}

We are going to show that the set of all $d$-integrals on a given semiring $\mathbf S=(S,+,\cdot,0,1)$ (together with the empty set in the case when $\Int_d(S)=S$) can be equipped with two binary operations such that the resulting structure forms a semiring again. A condition for this construction is that the set of all $d$-integrable elements of $S$ is closed under multiplication.

\begin{theorem}\label{th1}
Let $\mathbf S=(S,+,\cdot,0,1)$ be a semiring, $d\in\Der\mathbf S$ and
\[
I_d(S):=\{\emptyset\}\cup\{I_d(x)\mid x\in S\},
\]
assume
\begin{equation}\label{equ1}
\Int_d(S)\cdot\Int_d(S)\subseteq\Int_d(S)
\end{equation}
and define binary operations $+$ and $\cdot$ on $I_d(S)$ as follows:
\[
\begin{array}{lll}
I_d(x)\oplus I_d(y)                & :=I_d(x+y)      & \text{if }I_d(x),I_d(y)\neq\emptyset, \\
z\oplus\emptyset=\emptyset\oplus z & :=z,            & \\
I_d(x)\odot I_d(y)                 & :=I_d(x\cdot y) & \text{if }I_d(x),I_d(y)\neq\emptyset, \\
z\odot\emptyset=\emptyset\odot z   & :=\emptyset     &
\end{array}
\]
{\rm(}$x,y\in S$; $z\in I_d(S)${\rm)}. Then $(I_d(S),\oplus,\odot,\emptyset)$ is a {\rm(}not necessarily unitary{\rm)} semiring.
\end{theorem}

\begin{proof}
Let $a,b,c\in S$ and $e,f\in I_d(S)$. It is easy to see that $\oplus$ and $\odot$ are well-defined. Now (i) of Proposition~\ref{prop1} implies $\Int_d(S)+\Int_d(S)\subseteq\Int_d(S)$. If $a,b\in\Int_d(S)$ then
\[
I_d(a)\oplus I_d(b)=I_d(a+b)=I_d(b+a)=I_d(b)\oplus I_d(a).
\]
Moreover,
\begin{align*}
\emptyset\oplus e & =e=e\oplus\emptyset, \\
 e\oplus\emptyset & =e=\emptyset\oplus e.
\end{align*}
This shows that $\oplus$ is commutative. If $a,b,c\in\Int_d(S)$ then
\begin{align*}
\big(I_d(a)\oplus I_d(b)\big)\oplus I_d(c) & =I_d(a+b)\oplus I_d(c)=I_d\big((a+b)+c\big)=I_d\big(a+(b+c)\big)= \\
                                           & =I_d(a)\oplus I_d(b+c)=I_d(a)\oplus\big(I_d(b)\oplus I_d(c)\big).
\end{align*}
Moreover,
\begin{align*}
(\emptyset\oplus e)\oplus f & =e\oplus f=\emptyset\oplus(e\oplus f), \\
 (e\oplus\emptyset)\oplus f & =e\oplus f=e\oplus(\emptyset\oplus f), \\
 (e\oplus f)\oplus\emptyset & =e\oplus f=e\oplus(f\oplus\emptyset).
\end{align*}
This shows that $\oplus$ is associative. Because of
\[
e\oplus\emptyset=\emptyset\oplus e=e,
\]
$\emptyset$ is the neutral element with respect to $\oplus$. In order to be able to prove associativity of $\odot$ we need assumption (\ref{equ1}). If $a,b,c\in\Int_d(S)$ then
\begin{align*}
\big(I_d(a)\odot I_d(b)\big)\odot I_d(c) & =I_d(a\cdot b)\odot I_d(c)=I_d\big((a\cdot b)\cdot c\big)=I_d\big(a\cdot(b\cdot c)\big)= \\
                                         & =I_d(a)\odot I_d(b\cdot c)=I_d(a)\odot\big(I_d(b)\odot I_d(c)\big).
\end{align*}
We have
\begin{align*}
(\emptyset\odot e)\odot f & =\emptyset\odot f=\emptyset=\emptyset\odot(e\odot f), \\
 (e\odot\emptyset)\odot f & =\emptyset\odot f=\emptyset=e\odot\emptyset=e\odot(\emptyset\odot f), \\
 (e\odot f)\odot\emptyset & =\emptyset=e\odot\emptyset=e\odot(f\odot\emptyset).
\end{align*}
This shows that $\odot$ is associative. If $a,b,c\in\Int_d(S)$ then
\begin{align*}
\big(I_d(a)\oplus I_d(b)\big)\odot I_d(c) & =I_d(a+b)\odot I_d(c)=I_d\big((a+b)\cdot c\big)=I_d(a\cdot c+b\cdot c)= \\
                                          & =I_d(a\cdot c)\oplus I_d(b\cdot c)=\big(I_d(a)\odot I_d(c)\big)\oplus\big(I_d(b)\odot I_d(c)\big).																				
\end{align*}
Moreover,
\begin{align*}
(\emptyset\oplus e)\odot f & =e\odot f=\emptyset\oplus(e\odot f)=(\emptyset\odot f)\oplus(e\odot f), \\
 (e\oplus\emptyset)\odot f & =e\odot f=(e\odot f)\oplus\emptyset=(e\odot f)\oplus(\emptyset\odot f), \\
 (e\oplus f)\odot\emptyset & =\emptyset=\emptyset\oplus\emptyset=(e\odot\emptyset)\oplus(f\odot\emptyset).
\end{align*}
This proves one distributive law. The other one can be shown analogously. Finally,
\[
e\odot\emptyset=\emptyset\odot e=\emptyset.
\]
The proof is complete.
\end{proof}

It might be interesting to investigate how strong assumption (\ref{equ1}) of Theorem~\ref{th1} is. The following is immediate.

\begin{remark}
If $d(x)\cdot d(y)=0$ for all $x,y\in S$ then obviously {\rm(\ref{equ1})} is satisfied.
\end{remark}

Finally, let us see in which cases from our previous examples assumption (\ref{equ1}) holds.

\begin{example}
The semiring from Example~\ref{ex1} together with its derivation satisfies {\rm(\ref{equ1})} since $\left(
\begin{array}{cc}
0 & x \\
0 & 0
\end{array}
\right)\left(
\begin{array}{cc}
0 & y \\
0 & 0
\end{array}
\right)=\left(
\begin{array}{cc}
0 & 0 \\
0 & 0
\end{array}
\right)$ for all $x,y\in R$.
\end{example}

\begin{example}
The semiring from Example~\ref{ex2} together with its derivation satisfies {\rm(\ref{equ1})} if and only if ${\rm char}\mathbf F\in\{0,2\}$. This can be seen as follows:
\[
\Int_d(F[x])=\left\{
\begin{array}{ll}
F[x]                                                               & \text{if }{\rm char}\mathbf F=0, \\
\{\sum\limits_{i=0}^na_ix^i\mid a_i=0\text{ if }i\equiv1\bmod 2\}  & \text{if }{\rm char}\mathbf F=2, \\
\{\sum\limits_{i=0}^na_ix^i\mid a_i=0\text{ if }i\equiv-1\bmod p\} & \text{if }{\rm char}\mathbf F=p\geq3.
\end{array}
\right.
\]
Hence, if ${\rm char}\mathbf F=p\geq3$ then $x,x^{p-2}\in\Int_d(F[x])$, but $x^{p-1}=x\cdot x^{p-2}\notin\Int_d(F[x])$.
\end{example}

Authors' addresses:

Ivan Chajda \\
Palack\'y University Olomouc \\
Faculty of Science \\
Department of Algebra and Geometry \\
17.\ listopadu 12 \\
771 46 Olomouc \\
Czech Republic \\
ivan.chajda@upol.cz

Helmut L\"anger \\
TU Wien \\
Faculty of Mathematics and Geoinformation \\
Institute of Discrete Mathematics and Geometry \\
Wiedner Hauptstra\ss e 8-10 \\
1040 Vienna \\
Austria, and \\
Palack\'y University Olomouc \\
Faculty of Science \\
Department of Algebra and Geometry \\
17.\ listopadu 12 \\
771 46 Olomouc \\
Czech Republic \\
helmut.laenger@tuwien.ac.at
\end{document}